\documentclass[12pt]{article}
\usepackage{times}
\usepackage{amsfonts, amsmath, amssymb, amsthm, mathtools, dsfont}
\usepackage[colorlinks=true,citecolor=blue]{hyperref}
\usepackage[margin=2.5cm]{geometry}
\usepackage[utopia]{mathdesign}
\usepackage[mathscr]{euscript}
\usepackage[utf8]{inputenc}
\usepackage{tikz}
\usetikzlibrary{arrows}
\usetikzlibrary{decorations.pathreplacing,decorations.markings}

\tikzset{
  on each segment/.style={
    decorate,
    decoration={
      show path construction,
      moveto code={},
      lineto code={
        \path [#1]
        (\tikzinputsegmentfirst) -- (\tikzinputsegmentlast);
      },
      curveto code={
        \path [#1] (\tikzinputsegmentfirst)
        .. controls
        (\tikzinputsegmentsupporta) and (\tikzinputsegmentsupportb)
        ..
        (\tikzinputsegmentlast);
      },
      closepath code={
        \path [#1]
        (\tikzinputsegmentfirst) -- (\tikzinputsegmentlast);
      },
    },
  },
  mid arrow/.style={postaction={decorate,decoration={
        markings,
        mark=at position .5 with {\arrow[#1]{stealth}}
      }}},
}

\usepackage{graphicx}
\usepackage{parskip}
\usepackage{enumerate}
\usepackage{verbatim}
\usepackage{units}

\usetikzlibrary{calc}
\usepackage[outline]{contour}
\contourlength{1.2pt}

\begingroup
    \makeatletter
    \@for\theoremstyle:=definition,remark,plain\do{%
        \expandafter\g@addto@macro\csname th@\theoremstyle\endcsname{%
            \addtolength\thm@preskip\parskip
            }%
        }
\endgroup

\newtheorem{theorem}{Theorem}[section]
\newtheorem*{theorem*}{Theorem}
\newtheorem*{thm-dirac}{Dirac's Theorem}
\newtheorem*{thm-ore}{Ore's Theorem}
\newtheorem{lemma}[theorem]{Lemma}

\theoremstyle{definition}

\newtheorem*{remark*}{Remark}
\newtheorem{claim}[theorem]{Claim}

\begin{document}

\title{A sharp Ore-type condition for a connected graph with no induced star to have a Hamiltonian path}

\author{
{Ilkyoo Choi}\thanks{
Supported by the Basic Science Research Program through
the National Research Foundation of Korea (NRF) funded by the Ministry of
Education (NRF-2018R1D1A1B07043049), and also by Hankuk University of Foreign
Studies Research Fund.
Department of Mathematics, Hankuk University of Foreign Studies, Yongin-si,
Gyeonggi-do, Republic of Korea.
\texttt{ilkyoo@hufs.ac.kr}.
}
\and
Jinha Kim\thanks{
Department of Mathematical Sciences, Seoul National University, Seoul, Republic of Korea.
\texttt{kjh1210@snu.ac.kr}
}
}

\date\today

\maketitle

\begin{abstract}
We say a graph $G$ has a {\it Hamiltonian path} if it has a path containing all vertices of $G$. 
For a graph $G$, let $\sigma_2(G)$ denote the minimum degree sum of two nonadjacent vertices of $G$; 
restrictions on $\sigma_2(G)$ are known as {\it Ore-type conditions}. 
Given an integer $t\geq 5$, we prove that if a connected graph $G$ on $n$ vertices satisfies $\sigma_2(G)>{t-3\over t-2}n$, then $G$ has either a Hamiltonian path or an induced subgraph isomorphic to $K_{1, t}$. 
Moreover, we characterize all $n$-vertex graphs $G$ where $\sigma_2(G)={t-3\over t-2}n$ and $G$ has neither a Hamiltonian path nor an induced subgraph isomorphic to $K_{1, t}$. 
This is an analogue of a recent result by Mom\`ege~\cite{2018Momege}, who investigated the case when $t=4$. 
\end{abstract}

\section{Introduction}

Given a graph $G$, let $V(G)$ and $E(G)$ denote the vertex and edge set, respectively, of $G$.
Also, let $c(G)$ and $p(G)$ denote the number of vertices on a longest cycle and path, respectively, in $G$. 
We say a graph $G$ has a {\it Hamiltonian cycle} if $G$ has a cycle containing all vertices of $G$; in other words, $c(G)=|V(G)|$.
Similarly, a {\it Hamiltonian path} is a path containing all vertices. 

Since it is NP-complete to determine if Hamiltonian cycles exist in graphs, finding sufficient conditions for a graph to have a Hamiltonian cycle is of great interest, to the extent that it has formed an entire branch of extremal and structural graph theory. 
We dare not provide a summery of the vast history related to Hamiltonian cycles and paths, but we refer the readers to excellent surveys~\cite{1991Gould,2003Gould,2014Gould,2013Li}.
We will however state two celebrated results. 

In the literature, a {\it Dirac-type condition} typically refers to a lower bound on the minimum degree of (a vertex of) a graph.
This stems from the following classical result by Dirac, who proved a sharp lower bound on the minimum degree that guarantees a Hamiltonian cycle. 

\begin{thm-dirac}[\cite{1952Dirac}]\label{thm:dirac}
If $G$ is a graph on $n\geq 3$ vertices satisfying that each vertex has degree at least ${n\over 2}$, then $G$ has a Hamiltonian cycle. 
\end{thm-dirac}

Another important condition considered in the literature is an {\it Ore-type condition}, which is a lower bound on the minimum degree sum of two nonadjacent vertices.
As usual, let $\sigma_k(G)=\min\{\deg(x_1)+\cdots+\deg(x_k): x_1, \ldots, x_k \mbox{are pairwise nonadjacent}\}$;
therefore, an Ore-type condition is a lower bound on $\sigma_2$. 
Strengthening Dirac's Theorem above, Ore proved a sharp lower bound on $\sigma_2$ that guarantees a Hamiltonian cycle. 

\begin{thm-ore}[\cite{1960Ore}]\label{thm:ore}
If $G$ is a graph on $n\geq 3$ vertices satisfying $\sigma_2(G)\geq n$, then $G$ has a Hamiltonian cycle. 
\end{thm-ore}

It is not hard to deduce a sharp Dirac-type condition and a sharp Ore-type condition that guarantee Hamiltonian paths using Dirac's Theorem and Ore's Theorem.

Very recently, Mom\`ege~\cite{2018Momege} proved that there exists an Ore-type condition that guarantees a connected graph to have either a Hamiltonian path or an induced subgraph isomorphic to $K_{1, 4}$. 
The exact statement is the following:

\begin{theorem}[\cite{2018Momege}]\label{thm:k14}
If $G$ is a connected graph on $n$ vertices satisfying $\sigma_2(G)\geq {2\over 3}n$, then $G$ has either a Hamiltonian path or an induced subgraph isomorphic to $K_{1, 4}$.
\end{theorem}

Our main result is in the spirit of the aforementioned result.
For an integer $t\geq 5$, we prove a sharp Ore-type condition that guarantees the existence of either a Hamiltonian path or an induced subgraph isomorphic to $K_{1, t}$. 
Our main result is the following: 

\begin{theorem}\label{thm:main}
For an integer $t \geq 5$, if $G$ is a connected graph on $n$ vertices satisfying $\sigma_2(G) > \frac{t-3}{t-2}n$, then $G$ has either a Hamiltonian path or an induced subgraph isomorphic to $K_{1,t}$.
\end{theorem}

The threshold on the Ore-type condition is sharp, as one can see from the following example:
Fix $t\geq 5$, and let $H_t$ be a copy of $K_{t-3,t-1}$.
The graph $H_t$ is a connected graph on $2t-4$ vertices satisfying $\sigma_2(H_t) = 2(t-3) =\frac{t-3}{t-2}(2t-4)$.
Note that $H_t$ has neither a Hamiltonian path nor an induced subgraph isomorphic to $K_{1, t}$.

Furthermore, for $t\geq 5$, we actually show that all sharpness examples for Theorem~\ref{thm:main} must be the join of a graph on $t-3$ vertices and $t-1$ pairwise nonadjacent vertices; for more details, see Section~\ref{sec:remarks}. 
We end this section with two lemmas, which appeared in~\cite{2018Momege}, that will be used in our proof of Theorem~\ref{thm:main} in Section~\ref{sec:proof}.



\begin{lemma}[\cite{2018Momege}]\label{lem:sig}
Let $G$ be a graph on $n$ vertices and $1 \leq k \leq n-1$. Then $\sigma_{k+1}(G) \geq \frac{k+1}{k}\sigma_k(G)$.
\end{lemma}

\begin{lemma}[\cite{2018Momege}]\label{lem:dom}
Let $G$ be a connected graph on $n$ vertices with $\sigma_3(G) \geq n$ and $p(G) \leq n-1$. Then 
\begin{enumerate}[$(a)$]
\item $c(G) = p(G)-1$ 
\item If $C$ is a longest cycle of $G$, then for all $v$ in $V(G) \setminus V(C)$ we have $N_G(v) \subset V(C)$.
\end{enumerate}
\end{lemma}

\section{Proof of Theorem~\ref{thm:main}}\label{sec:proof}

Let $G$ be a connected graph on $n$ vertices satisfying $\sigma_2(G) > {t-3\over t-2} n$ for some integer $t\geq 5$.
Assume $G$ has no Hamiltonian path, so our goal is to prove that $G$ has an induced subgraph isomorphic to $K_{1, t}$. 
Let $C=c_1c_2\ldots c_mc_1$ be a longest cycle of $G$.
By Lemma~\ref{lem:sig}, we know $\sigma_3(G)>n$ since $t\geq 5$ and $\sigma_3(G) \geq \frac{3}{2} \sigma_2(G) > {3\over 2}\cdot{2\over 3}n =n$.
Now Lemma~\ref{lem:dom} is applicable to $G$, so $p(G)=m+1$ and vertices not on $C$ are pairwise nonadjacent. 
Since $G$ has no Hamiltonian path, we know $m\leq n-2$, so $G$ has at least two vertices not on $C$; let $u$ and $v$ be two such vertices. 
Since $\deg_G(u)+\deg_G(v) \geq \sigma_2(G) > {t-3\over t-2} n$, without loss of generality, we may assume $\deg_G(v) > {t-3\over 2(t-2)}{n}$.

We will abuse notation and say addition on the indicies of vertices on $C$ are modulo $m$; for example, $c_0=c_m$ and $c_{m+1}=c_1$.
Define $I(v):=\{i \in \{1,2,\dots,m\} : vc_{i},vc_{i+2} \in E(G)\}$.

We now prove a sequence of claims. 

\begin{claim}\label{claim:cycleedge} 
For an integer $i$, $c_i$ and $c_{i+1}$ cannot both have neighbors not on $C$. 
\end{claim}
\begin{proof}
Let $x, y\notin V(C)$ such that $xc_i, yc_{i+1}\in E(G)$. 
Consider $C'=C+xc_{i}+yc_{i+1}-c_{i}c_{i+1}$.
If $x=y$, then $C'$ is a cycle that is longer than $C$, which is a contradiction. 
If $x\neq y$, then $C'$ is a path on $m+2$ vertices, which contradicts $p(G)=m+1$. 
See Figure~\ref{fig:cycleedge}.
\begin{figure}[htbp]
    \centerline{\includegraphics[scale=1]{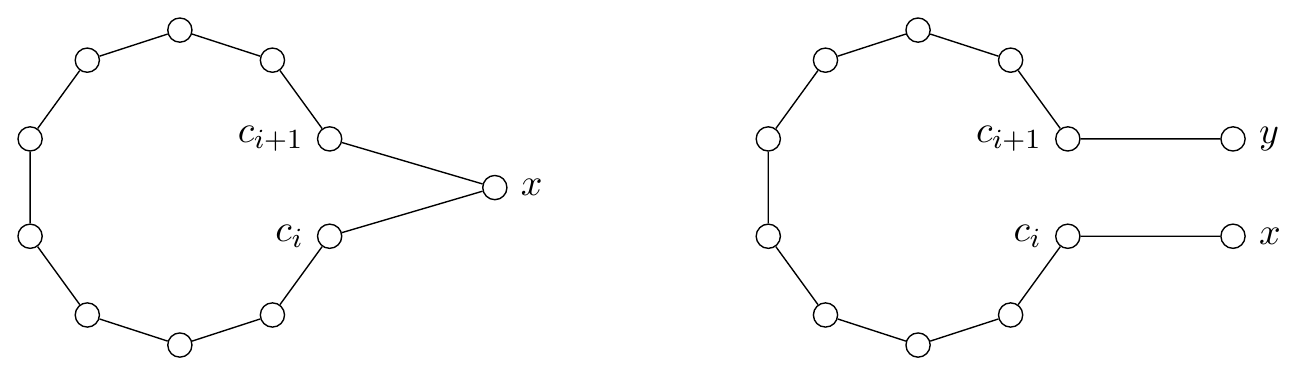}}
    \caption{$c_i$ and $c_{i+1}$ cannot both have neighbors not on $C$.}
    \label{fig:cycleedge}
\end{figure}
\end{proof}

\begin{claim}\label{claim:chord} 
For integers $i$ and $j$ where $|i-j|>1$, if $c_ic_j\in E(G)$, then $c_{i+1}$ and $c_{j+1}$ cannot both have neighbors not on $C$. 
\end{claim}
\begin{proof}
Let $x,y\notin V(C)$ such that $xc_{i+1}, yc_{j+1}\in E(G)$. 
Consider $C'=C+xc_{i+1}+yc_{j+1}+c_ic_j-c_ic_{i+1}-c_jc_{j+1}$.
If $x=y$, then $C'$ is a cycle that is longer than $C$, which is a contradiction. 
If $x\neq y$, then $C'$ is a path on $m+2$ vertices, which contradicts $p(G)=m+1$.
See Figure~\ref{fig:chord}.
\begin{figure}[htbp]
    \centerline{\includegraphics[scale=1]{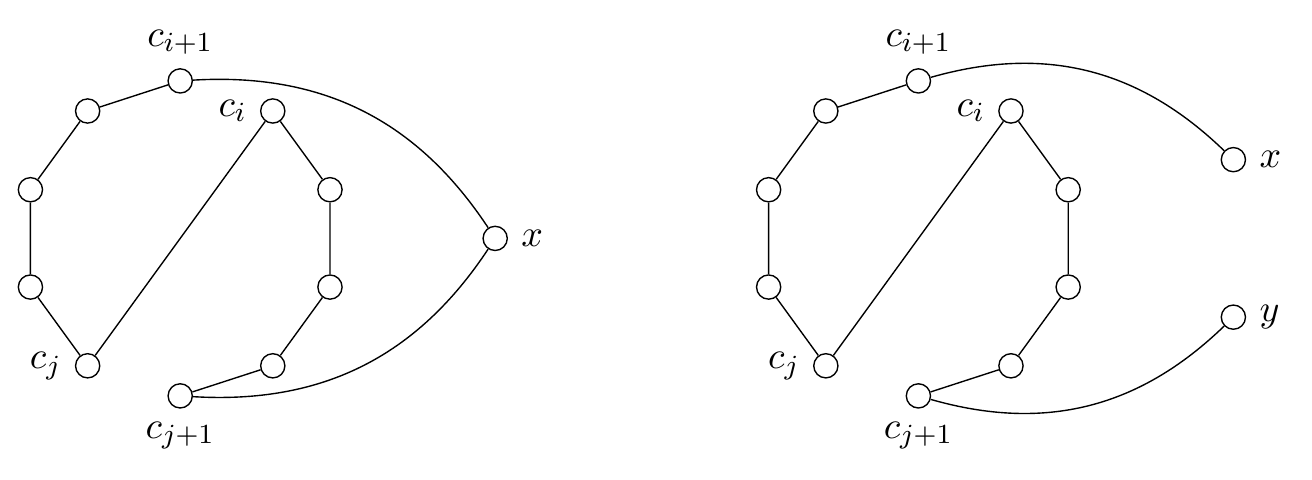}}
    \caption{If $c_ic_j\in E(G)$, then $c_{i+1}$ and $c_{j+1}$ cannot both have neighbors not on $C$.}
    \label{fig:chord}
\end{figure}
\end{proof}

\begin{claim}\label{claim:sizei}
$|I(v)|>t-3$.
\end{claim}
\begin{proof}
We will first show that $n>2t-4$.
Suppose to the contrary that $n \leq 2t-4$.
By Claim~\ref{claim:cycleedge}, we know $\deg_G(v) \leq \frac{|V(C)|}{2}$ and $\deg_G(u) \leq \frac{|V(C)|}{2}$.
Now, $\sigma_2(G) \leq \deg_G(u)+\deg_G(v) \leq |V(C)| \leq n-2$.
On the other hand, since $\sigma_2(G) > \frac{t-3}{t-2}n \geq n-2$ when $n\leq 2t-4$ and $\sigma_2(G)$ is an integer, it follows that $\sigma_2(G) \geq n-1$, which is a contradiction.

Now, let us count the number of neighbors of $v$, which we know are all on $C$.
By Claim~\ref{claim:cycleedge}, there are $|N_G(v)|$ vertices on $C$ that cannot be neighbors of $v$. 
Also, if $c_i \in N_G(v)$ but $i \notin I(v)$, then $c_{i+2} \notin N_G(v)$ by definition of $I(v)$. 
Hence, $|N_G(v)| \leq |V(C)|-|N_G(v)|-\left(|N_G(v)|-|I(v)|\right)$.
Thus, $|I(v)| \geq 3|N_G(v)|-|V(C)| > \frac{3(t-3)}{2(t-2)}n-(n-2)=\left(\frac{3(t-3)}{2(t-2)}-1\right)n+2>t-3$.
\end{proof}


\begin{claim}\label{claim:t-4}
For some $c_l \in N_G(u) \cap N_G(v)$, there exist $t-4$ distinct elements $j_1,j_2,\dots,j_{t-4} \in I(v)$ such that $c_{j_1+1},\dots,c_{j_{t-4}+1} \in N_G(c_l)-\{c_{l-1},c_{l+1}\}$.
\end{claim}
\begin{proof}
Suppose to the contrary that $|\{j \in I(v) : c_{j+1} \in N_G(c_l)-\{c_{l-1},c_{l+1}\}\}| \leq t-5$ for all $c_l \in N_G(u) \cap N_G(v)$.
For a particular $i \in I(v)$, let us count the number of neighbors of $c_{i+1}$.
Note that by Claim~\ref{claim:cycleedge}, all neighbors of $c_{i+1}$ are on $C$. 
By Claim~\ref{claim:chord}, $c_{i+1}$ cannot be adjacent to $|N_G(u)\cup N_G(v)|-1$ vertices on $C$.
Therefore, we obtain the following inequalities:
\begin{equation*}
\begin{aligned}
|N_G(c_{i+1}) \setminus (N_G(u) \cap N_G(v)-\{c_i,c_{i+2}\})| 
&\leq 2+|C|-3-(|N_G(u) \cap N_G(v)|-2)-(|N_G(u) \cup N_G(v)|-1)\\
&=2+|C|-(|N_G(u)|+|N_G(v)|)\\
&\leq n-\sigma_2(G)\\
&<n-\frac{t-3}{t-2}n=\frac{n}{t-2}.
\end{aligned}
\end{equation*}

\begin{equation*}
\begin{aligned}
|N_G(c_{i+1})| 
&\leq |N_G(c_{i+1}) \setminus (N_G(u) \cap N_G(v)-\{c_i,c_{i+2}\})|+|N_G(c_{i+1}) \cap (N_G(u) \cap N_G(v)-\{c_i,c_{i+2}\})|\\
&<\frac{n}{t-2}+|(N_G(c_{i+1})-\{c_i,c_{i+2}\}) \cap N_G(u) \cap N_G(v)|.
\end{aligned}
\end{equation*}

Let $i_1,\dots,i_{t-2}$ be distinct elements in $I(v)$, which we know exist by Claim~\ref{claim:sizei}.
Note that $c_l \in N_G(c_{i_j+1})-\{c_{i_j},c_{i_j+2}\}$ if and only if $c_{i_j+1} \in N_G(c_l)-\{c_{l-1},c_{l+1}\}$.
We obtain the following inequality: 
\begin{equation*}
\begin{aligned}
\deg_G(c_{i_1+1})+\dots+\deg_G(c_{i_{t-2}+1}) 
&< (t-2)\cdot \frac{n}{t-2} +\sum_{j=1}^{t-2} |(N_G(c_{i_j+1})-\{c_{i_j},c_{i_j+2}\}) \cap N_G(u) \cap N_G(v)|\\
&=n+\sum_{c_l \in N_G(u) \cap N_G(v)} |\{j \in \{i_1,\dots,i_{t-2}\}:c_l \in N_G(c_{i_j+1})-\{c_{i_j},c_{i_j+2}\}\}|\\
&= n+\sum_{c_l \in N_G(u) \cap N_G(v)} |\{j \in \{i_1,\dots,i_{t-2}\}:c_{i_j+1} \in N_G(c_l)-\{c_{l-1},c_{l+1}\}\}|\\
&\leq n+(t-5)|N_G(u) \cap N_G(v)| \\
&\leq n+(t-5)|N_G(v)| \\
&\leq n+(t-5)\frac{|V(C)|}{2} \\
&\leq n+\frac{t-5}{2}(n-2)=\frac{t-3}{2}n-(t-5).
\end{aligned}
\end{equation*}

Note that $c_{i_1+1},\dots,c_{i_{t-2}+1}$ are pairwise nonadjacent in $G$ by Claim~\ref{claim:chord}.
Now, by Lemma~\ref{lem:sig},
\[
\deg_G(c_{i_1+1})+\deg_G(c_{i_2+1})+\dots+\deg_G(c_{i_{t-2}+1}) \geq \sigma_{t-2}(G) \geq \frac{t-2}{2}\sigma_2(G) >\frac{t-3}{2}n.
\]
This contradicts the above inequality. 
\end{proof}


We now finish the proof by explicitly finding an induced subgraph isomorphic to $K_{1, t}$. 
Fix a vertex $c_l \in N_G(u) \cap N_G(v)$, and let $j_1,j_2,\dots,j_{t-4}$ be $t-4$ distinct elements in $I(v)$ such that $c_{j_1+1},\dots,c_{j_{t-4}+1} \in N_G(c_l)-\{c_{l-1},c_{l+1}\}$; 
these elements are guaranteed to exist by Claim~\ref{claim:t-4}.
We will show that the induced subgraph of $G$ on $u,v,c_{l-1},c_l,c_{l+1},c_{j_1+1},\dots,c_{j_{t-4}+1}$ is isomorphic to $K_{1,t}$; 
the vertex of degree $t$ is $c_l$.

\begin{itemize}
\item $c_lu,c_lv,c_lc_{l-1},c_lc_{l+1},c_lc_{j_1+1},c_lc_{j_2+1},\dots,c_lc_{j_{t-4}+1} \in E(G).$

Follows from the choice of $c_l$ and the choice of indices $j_1, \ldots, j_{t-4}$. 

\item $uv \notin E(G)$

Follows from the fact that $N_G(v)\subset V(C)$. 

\item $zc_q\notin E(G)$ for $z\in\{u, v\}$ and $q\in\{l-1, l+1, j_1+1, \ldots, j_{t-4}+1\}$.

Follows from Claim~\ref{claim:cycleedge} since $z$ is adjacent to $c_l, c_{j_1}, \ldots, c_{j_{t-4}}$.

\item $c_{l-1}c_{l+1} \notin E(G)$

If $c_{l-1}c_{l+1} \in E(G)$, then $G$ has a cycle $C+c_{l-1}c_{l+1}+c_lc_{j_1+1}+vc_{j_1}+vc_l-c_{j_1}c_{j_1+1}-c_{l-1}c_l-c_lc_{l+1}$ on $m+1$ vertices, which is a contradiction. 
See Figure~\ref{fig:help1}.

\begin{figure}[htbp]
    \centerline{\includegraphics[scale=1]{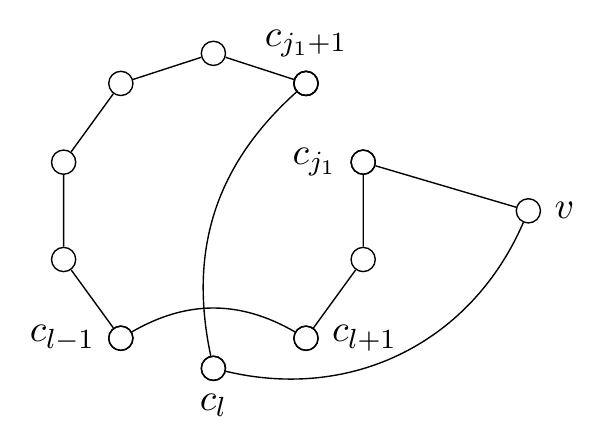}}
    \caption{A cycle that is longer than $C$.}
    \label{fig:help1}
\end{figure}

\item $c_pc_q\notin E(G)$ for $p\in\{l-1, l+1\}$ and $q\in\{j_1+1, \ldots, j_{t-4}+1\}$.

We will argue only $c_{l-1}c_{j_1+1}, c_{l+1}c_{j_1+1}\notin E(G)$, as the arguments for other cases are identical.
The former follows from Claim~\ref{claim:chord} since $v$ is adjacent to $c_l$ and $c_{j_1+2}$. 
The latter follows from Claim~\ref{claim:chord} since $v$ is adjacent to $c_l$ and $c_{j_1}$. 

\item $c_pc_q\notin E(G)$ for $p, q\in\{j_1+1, \ldots, j_{t-4}+1\}$. 

Follows from Claim~\ref{claim:chord} since $v$ is adjacent to both $c_{p+1}$ and $c_{q+1}$.

\end{itemize}

Therefore, $G$ has an induced subgraph isomorphic to $K_{1,t}$.
This concludes the proof of Theorem~\ref{thm:main}.


\section{Remarks}\label{sec:remarks}

We can actually strengthen Theorem~\ref{thm:main} as in the following.
We use $G\vee H$ to denote the join of two graphs $G$ and $H$, and use $\overline{K}_{t-1}$ to denote $t-1$ pairwise nonadjacent vertices. 

\begin{theorem}
For $t\geq 5$, let $G$ be a connected graph on $n$ vertices satisfying $\sigma_2(G)\geq{t-3\over t-2}n$.
If $G$ has no Hamiltonian path, then $n\geq 2t-4$, and moreover, 
\begin{enumerate}
\item if $n> 2t-4$, then $G$ has an induced subgraph isomorphic to $K_{1, t}$.
\item if $n=2t-4$, then $G$ is of the form $H \vee \overline{K}_{t-1}$ where $H$ is a graph on $t-3$ vertices. 
\end{enumerate}
\end{theorem}

\begin{proof}
If $n<2t-4$, then $\sigma_2(G)\geq\frac{t-3}{t-2}n=\left(1-{1\over t-2}\right)n>\left(1-{1\over n/2}\right)n=n-2$.
Since $\sigma_2(G)$ is an integer, this implies $\sigma_2(G)\geq n-1$. 
This further implies $G$ has a Hamiltonian path, which is a contradiction.

If $n>2t-4$, then one can check that the exact arguments of Theorem~\ref{thm:main} hold to conclude that $G$ has an induced subgraph isomorphic to $K_{1, t}$.

Now assume $n=2t-4$.
Since $\sigma_2(G)\geq {t-3\over t-2}n>{(t-1)-3\over (t-1)-2}n$ and $\sigma_2(G)\geq {2\over 3}n$, by Theorem~\ref{thm:main} and by Theorem~\ref{thm:k14}, $G$ has an induced subgraph $S$ isomorphic to $K_{1,t-1}$.
Let $x_1,\ldots,x_{t-1}$ be the vertices of $S$ that have degree 1 in $S$. 
Since $x_1,\ldots,x_{t-1}$ are pairwise nonadjacent in $G$, we know each $x_i$ has at most $n-(t-1)=t-3$ neighbors.
Since $2t-6=\frac{t-3}{t-2}n = \sigma_2(G) \leq \deg_G(x_i)+\deg_G(x_j) \leq 2t-6$ for each distinct pair $x_i,x_j$, we obtain that $\deg_G(x_i) =t-3$ for $i\in\{1, \ldots, t-1\}$. 
Note that $V(S)-\{x_1, \ldots, x_{t-1}\}$ has exactly $t-3$ vertices. 

Therefore, $G$ is of the form $H\vee \overline{K}_{t-1}$ where $H$ is a graph on $t-3$ vertices. 
Moreover, regardless of $H$, it is not hard to see that $G$ satisfies $\sigma_2(G)={t-3\over t-2}n$ and has neither a Hamiltonian path nor an induced subgraph isomorphic to $K_{1, t}$.
\end{proof}




\end{document}